\documentclass[11pt]{amsart}

\usepackage{amsmath, amsfonts,amssymb,amsthm,amstext,amscd}
\usepackage{tikz}  
\usepackage{enumerate}
\usepackage{xcolor}
\usepackage{tkz-graph}

\usetikzlibrary{decorations.pathmorphing}

\tikzstyle{vertex}=[circle, draw, inner sep=0pt, minimum size=6pt]
\newcommand{\vertex}{\node[vertex]}

\theoremstyle{plain}
\newtheorem{theorem}{Theorem}[section]
\newtheorem{lemma}[theorem]{Lemma}
\newtheorem{corollary}[theorem]{Corollary}

\theoremstyle{definition}

\def \R {\mathbb{R}}
\def \C {\mathbb{C}}

\def \Z {\mathbb{Z}}

\def \GF {{\mbox {GF}(2)}}

\begin{document}

\title{Equivalence of edge bicolored graphs on surfaces}
\date{\today}

\author[O. T. Dasbach]{Oliver T. Dasbach}
 \address{Department of Mathematics, Louisiana State University,
 Baton Rouge, LA 70803}
\email{dasbach@math.lsu.edu}
\thanks {The first author was supported in part by NSF grant DMS-1317942.}
 
\author[H. M. Russell]{Heather M. Russell}
\address{Department of Mathematics\\Jepson Hall\\University of Richmond, Richmond, VA 23173}
\email{hrussell@richmond.edu}

\begin{abstract} 
Consider the collection of edge bicolorings of a graph that is cellularly embedded on an orientable surface. In this work, we count the number of equivalence classes of such colorings under two relations: reversing colors around a face and reversing colors around a vertex. In the case of the plane, this is well studied, but for other surfaces, the computation is more subtle. While this question can be stated purely graph theoretically, it has interesting applications in knot theory.
\end {abstract}
 
\maketitle

\def \R {\mathbb{R}}
\def \C {\mathbb{C}}
\def \Z {\mathbb{Z}}

\section{Introduction}

Consider the following game for a connected, finite graph cellularly embedded (i.e. every face is a disk) on a compact orientable surface. Given a coloring of the edges by two colors, a finite number of two kinds of moves are allowed: (1) Reverse the colors of all the edges connected to some vertex, and (2) Reverse the colors of all the edges bounding some face. The goal is to find the number of equivalence classes of colorings under these moves. 

This problem arises in the context of knot theory. Given a checkerboard coloring of a link diagram lying on the sphere or more generally on a compact orientable surface, a checkerboard graph of the diagram has a natural edge coloring coming from crossing data. Our moves on this graph correspond exactly to region crossing changes which reverse all crossings bordering some (black or white) region of the diagram. 

Recent work of Ayaka Shimizu \cite{Shimizu:RegionChanges} shows that every knot diagram on a sphere can be transformed to one of the unknot by a sequence of region crossing changes. This is proven directly on the level of knot diagrams.  Cheng Zhiyun and Gao Hongzhu \cite{ChengGao:RegionChanges} investigate region crossing changes for links of two components using graph theory and linear algebra. Cheng Zhiyun \cite{Cheng:RCCLinks} extends this work to links of $n$ components giving necessary and sufficient linking number conditions for a link to lie in the same equivalence class as the unlink. 

Every graph cellularly embedded on the sphere is the checkerboard graph of a knot or link, so the results of Shimizu and Cheng-Gao answer our question in the genus zero case. However, we will show that it also follows from previously known results. The number of equivalence classes is given by the absolute value $|T(-1,-1)|$ of the Tutte polynomial $T(x,y)$ at $(x,y)=(-1,-1)$ \cite{RosenstiehlRead:EdgeTripartition}. When a graph arises as a checkerboard graph of a diagram of a link $L$, the absolute value $|T(-1,-1)|$ equals the absolute value of the Jones polynomial of $L$ at $1$. The number of equivalence classes is given by $2^{c-1}$ where $c$ is the number of components of $L$. Moreover, representatives for the equivalence classes can be read off from $L$.

For graphs on orientable surfaces of higher genus, one might hope that the equivalence classes are counted by the generalization of the Tutte polynomial to graphs on surfaces: the Bollob\'as-Riordan-Tutte (BRT) polynomial \cite{BR:BRTorientable, DFKLS:GraphsOnSurfaces}. We will show the BRT polynomial does indeed count the components of a link on a surface. However, it turns out that the  problem of counting equivalence classes does not just depend on the genus of the surface and an evaluation of the BRT polynomial of the graph. Instead, the number of equivalence classes can be computed by finding the kernel of a certain map on homology.

We start with a motivating example of a graph embedded on a torus in Section \ref{section:example}. Section \ref{section:medial graph} looks at the connection between graphs and links defining a subspace of the intersection of the cycle space of the graph and its dual whose dimension is measured by the Bollob\'as-Riordan-Tutte polynomial. In Section \ref{section:plane} we discuss the game for plane graphs, and Section \ref{section:surfaces} covers the case for graphs on orientable compact surfaces of arbitrary genus. 

We have given an intuitive description of the two moves of the game in this introduction. The precise description of the moves, given in the statement of Theorem \ref{Thm:PlaneGraphs}, is slightly more complicated in order to allow for possible loops and bridges in $G$. Note that Zaslavsky \cite{Zaslavsky:SignedGraphs} investigates a related problem for signed graphs called switching equivalence in which only one of our two moves is allowed.

\subsection{Acknowledgement: } We are grateful to Adam Lowrance, Allison Henrich, Neal Stoltzfus, and Sergei Chmutov for numerous discussions and suggestions, and to Dan Silver for introducing us to the bicycle space.

\section{Example} \label{section:example}

\begin{figure}
\begin {center}
\begin{tikzpicture}
 \GraphInit[vstyle=Art]
\draw [dashed] (0,0) -- (6,0) -- (6,6) -- (0,6) -- (0,0);
\vertex (p1) at (2,2) {};
\vertex (p2) at (2,4) {};
\vertex (p3) at (4,2) {};
\vertex (p4) at (4,4) {};
\Edge [label=$2$, labelstyle=left](p1)(p2);
\Edge [label=$3$, labelstyle=above](p2)(p4);
\Edge [label=$4$, labelstyle=left](p4)(p3);
\Edge [label=$1$, labelstyle=above](p3)(p1);
\Edge [label=$7$, labelstyle=above](0,2)(p1);
\Edge [label=$5$, labelstyle=left](p1)(2,0);
\Edge [label=$8$, labelstyle=above](p2)(0,4);
\Edge [label=$5$, labelstyle=left](p2)(2,6);
\Edge [label=$6$, labelstyle=left](p3)(4,0);
\Edge [label=$7$, labelstyle=above](p3)(6,2);
\Edge [label=$6$, labelstyle=left](p4)(4,6);
\Edge [label=$8$, labelstyle=above](p4)(6,4);
\Edge [label=$9$, labelstyle={above=6pt}](p1)(p4);
\end{tikzpicture}
\caption{Graph with four vertices and nine edges on a torus \label{fig:TicTacToeGraph}}
\end{center}
\end{figure}
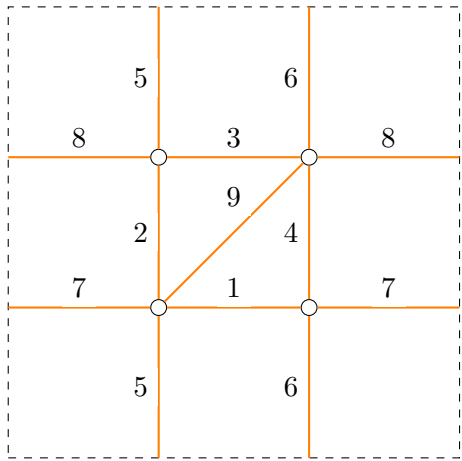

We discuss the mathematics of the game in an example. 
Figure \ref{fig:TicTacToeGraph} shows a graph $G=(V,E)$ with four vertices and eight edges embedded on a plane model of a torus. The torus is obtained by identifying the top with the bottom dotted edge of the square  and the left with the right dotted edge. 

We translate the game into a problem of determining the dimension of a vector space over $\GF = \mathbb{Z}/2\mathbb{Z}$, the field with two elements $0$ and $1$. We fix an ordering of the edges of $G$ (and hence $G^*$) and use it implicitly. Assigning, say, $0$ to the color blue and $1$ to the color red, we see that bicolorings of $G$ are in bijection with vectors in $\GF^{|E|}$. Given some coloring vector $w\in \GF^{|E|}$, switching colors around a vertex of $G$ corresponds to adding a row of the incidence matrix of $G$ to $w$. The row vectors of the incidence matrix form a subspace $U$ of $\GF^{|E|}$ of dimension $|V|-1$ (see e.g. \cite{GodsilRoyle:AlgebraicGraphTheory}, Lemma 14.15.1). This subspace $U$ is called the cocycle space (or cut space) of $G$.

For our example graph $G$ in Figure \ref{fig:TicTacToeGraph}, the incidence matrix $\mathcal I$ over $\GF$  is

$${\mathcal I}=\left(
\begin{array}{ccccccccc}
 1 & 1 & 0 & 0 & 1 & 0 & 1 & 0 & 1\\
 1 & 0 & 0 & 1 & 0 & 1 & 1 & 0 & 0\\
 0 & 0 & 1 & 1 & 0 & 1 & 0 & 1 & 1\\
 0 & 1 & 1 & 0 & 1 & 0 & 0 & 1 & 0\\
\end{array}
\right).
$$
Since the graph has $|V|=4$ vertices the rank over $\GF$ is $|V|-1=3$. Any row is linearly dependent on the other rows, so the cocycle space $U$ is the subspace of $\GF^9$ generated by the first three rows of $\mathcal I$.

The graph has five regions in its embedding, thus the dual graph $G^* = (V^*,E^*)$ has five vertices. The incidence matrix of the dual graph in this case is 
$${\mathcal I}^*=\left(
\begin{array}{cccccccccc}
 1 & 0 & 0 & 1 & 0 & 0 & 0 & 0 & 1 \\
 0 & 1 & 1 & 0 & 0 & 0 & 0 & 0 & 1 \\
 1 & 0 & 1 & 0 & 1 & 1 & 0 & 0 & 0 \\
 0 & 1 & 0 & 1 & 0 & 0 & 1 & 1 & 0 \\
 0 & 0 & 0 & 0 & 1 & 1 & 1 & 1 & 0 \\
\end{array}
\right).
$$
The edges in $E$ and $E^*$ are in $1$-$1$ correspondence, and thus $|E|=|E^*|$. Here, and throughout our discussion, we pick the same ordering on $E$ and $E^*$ and use this when considering rows of both ${\mathcal{I}}$ and ${\mathcal{I}}^*$ as vectors in $\GF^{|E|}$. 

Adding rows of ${\mathcal I}^*$ to a coloring vector corresponds to switching colors around faces of the graph. Since the dual graph has five vertices, its cocycle space $U^*$ is $4$-dimensional. Any vertex and face in $G$ are incident to either two or zero common edges. This means $U^*$ is orthogonal to $U$, and thus $U^* \subseteq U^{\perp}.$ The subspace $U^{\perp} \subseteq \GF^{|E|}$ is called the cycle space (or flow space) of $G$. The dimensions of the cycle and cocycle spaces are related (e.g. \cite{GodsilRoyle:AlgebraicGraphTheory}) by the formula
$$\dim U + \dim U^{\perp} = |E|.$$
The intersection $U \cap U^{\perp}$ is called the bicycle space of $G$. Its dimension $b$ is given by the Tutte polynomial at $(-1,-1)$ \cite{RosenstiehlRead:EdgeTripartition}: 
$$|T_G(-1,-1)|=2^b.$$
(We give a formula for the Tutte polynomial via a specialization of the Bollob\'as-Riordan-Tutte polynomial in the next section.)

For our problem, we are interested in the dimension of $\GF^{9}/(U + U^*)$. Since $$\dim (U+U^*)= \dim U + \dim U^*-\dim U \cap U^*$$ we need to determine the dimension of $U \cap U^*$. We know that $U^* \subseteq U^{\perp}$, and thus the space $U \cap U^*$ is a subspace of the bicycle space. For our example in Figure \ref{fig:TicTacToeGraph},  the dimension of $U+U^*$ turns out to be $6$, and hence the dimension of $U \cap U^*$ is $1$.
Since the codimension of $U+U^*$ in $\GF^9$ is $3$ there are $2^3$ different equivalence classes for our game.

The construction can be interpreted within the setting of the balanced overlaid Tait (BOT) graph (e.g. \cite{CDR:Dimers,RussellEtAl:DehnColoring}):
For the graph $G$ and its dual graph $G^*$ chose a vertex $v_0$ in $G$ and an adjacent vertex $v_0^*$ in $G^*$.
Let $\mathcal I$ and $\mathcal I^*$ be the incidence matrices of graph $G$ and its dual over $\GF$. In the matrices we  remove a row corresponding to $v_0$ in $G$ and to $v_0^*$ in $G^*$. We form a new matrix $A$ by combining the rows of $\mathcal I$ and $\mathcal I^*$ excluding the two removed rows. This matrix can be interpreted in a different way. 

If we lay $G$ and $G^*$ on top of each other by inserting a new vertex for every intersection of an edge in $G$ with the corresponding edge in $G^*$, this forms the BOT graph. We remove $v_0$ and $v_0^*$ and all of its adjacent edges. By construction we obtain a tripartite graph. Figure \ref{fig:BOT} gives an example. The matrix $A$ describes the adjacencies of vertices in the BOT graph among the bipartition of the vertices. 

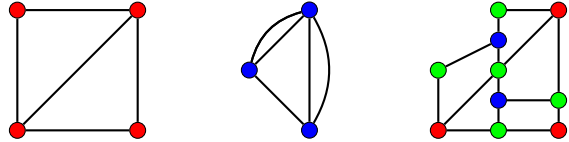
\begin{figure}[h]
\begin {center}
\begin{tikzpicture}[scale=0.4]
\vertex[fill=red] (p1) at (0,0) {};
\vertex[fill=red] (p2) at (4,0) {};
\vertex[fill=red] (p3) at (4,4) {};
\vertex[fill=red] (p4) at (0,4) {};
\Edges (p1,p2,p3,p4,p1);
\Edge (p1)(p3);
\end{tikzpicture}
\hspace {1cm}
\begin{tikzpicture}[scale=0.4]
\vertex[fill=blue] (q1) at (3,1) {};
\vertex[fill=blue] (q2) at (1,3) {};
\vertex[fill=blue] (q3) at (3,5) {};
\Edge (q1)(q2);
\Edge[style={bend left}](q2)(q3);
\Edge (q2)(q3);
\Edge[style={bend left}](q2)(q3)
\Edge (q1)(q3)
\Edge[style={bend right}](q1)(q3)
\end{tikzpicture}
\hspace {1cm}
\begin{tikzpicture}[scale=0.4]
\vertex[fill=red] (p1) at (0,0) {};
\vertex[fill=red] (p2) at (4,0) {};
\vertex[fill=red] (p3) at (4,4) {};
\vertex[fill=blue] (q1) at (2,1){};
\vertex[fill=blue] (q2) at (2,3){};
\vertex[fill=green] (r1) at (2,0){};
\vertex[fill=green] (r2) at (4,1){};
\vertex[fill=green] (r3) at (2,2){};
\vertex[fill=green] (r4) at (2,4){};
\vertex[fill=green] (r5) at (0,2){};
\Edges(p1, r1, p2, r2, p3, r4, q2, r3, q1);
\Edges(r1, q1, r2);
\Edges(p1,r3,p3);
\Edges(p1,r5,q2);

\end{tikzpicture}
\caption{Graph $G$, its dual graph, and the overlay graph with two vertices and their adjacent edges removed.  \label{fig:BOT}}
\end{center}
\end{figure}

\section{The components of the medial graph of a graph on a surface}
\label{section:medial graph}

Consider a link projection on an orientable surface such that all faces of the projection are disks and one can color the faces of the projection with black and white in a checkerboard fashion. Two dual checkerboard graphs $G_b$ and $G_w$ that are embedded on the surface are constructed as follows: The vertices of $G_b$ (resp. $G_w$) correspond to the faces colored in black (resp. white), and two vertices are connected by an edge if and only if the corresponding faces are adjacent to a common crossing of the link projection:

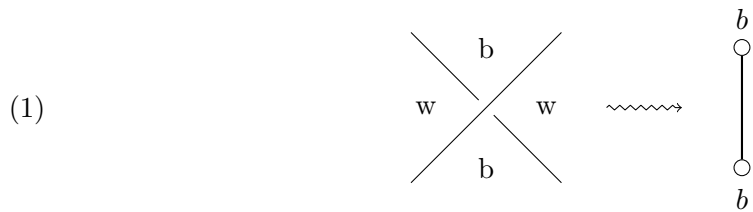
\begin{center}
\begin{equation} \label{CrossingToEdge}
\begin{tikzpicture}[baseline=25]
\draw (0,0)--(2,2);
\draw (1,1.8) node{b};
\draw (1,0.2) node{b};
\draw (1.8,1) node{w};
\draw (0.2,1) node{w};
\draw (2,0)--(1.1,0.9);
\draw (0.90, 1.1)--(0,2);
\begin{scope}[xshift=2.6cm, baseline=25.4]
\draw [->,
line join=round,
decorate, decoration={
    zigzag,
    segment length=4,
    amplitude=.9,post=lineto,
    post length=2pt
}]  (0,1) -- (1,1);
\end{scope}
\begin {scope}[xshift=3.4cm, baseline = 25]
\vertex [label={$b$}](p1) at (1,1.8) {};
\vertex [label={[yshift=-0.8cm] $b$}](p2) at (1,0.2) {};
\Edge(p1)(p2);
\draw (p1) -- (p2);
\end{scope}
\end{tikzpicture}
\end{equation}
\end{center}

Neither $G_b$ nor $G_w$ can be used to recover the original link diagram since crossing changes do not change the two checkerboard graphs. However, for any graph $G$ with a cellular embedding on a surface one can construct a unique alternating link diagram $D_L$ such that $G$ is one of the checkerboard graphs by reversing the arrow in (\ref{CrossingToEdge}). The link $L$ is called a {\it medial link} of $G$, and its underlying $4$-valent graph is called the {\it medial graph} of $G$. Since the medial graph comes from flattening a link diagram, it can be viewed as a collection of closed curves on the surface. In this sense, the number of components of the medial graph is well-defined. Figure \ref{Medial Graph} gives an example where the number of components is $3$.

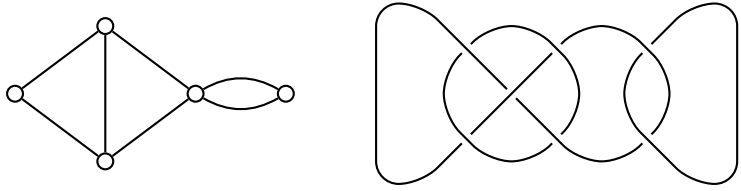
\begin{figure}[h]
\begin {center}
\begin{tikzpicture} [thick,scale = .6]
\vertex (p1) at (0,0){};
\vertex (p2) at (2,-1.5){};
\vertex (p3) at (2,1.5){};
\vertex (p4) at (4,0){};
\vertex (p5) at (6,0){};
\Edges (p1, p2,p3,p1);
\Edges (p2, p4, p3);
\Edge[style={bend right}](p4)(p5);
\Edge[style={bend right}](p5)(p4)

\begin{scope}[xshift = 9cm,yshift=-2cm, rounded corners = 3mm ]
\draw (.9,.9) -- (0,0) -- (-1,0) -- (-1,4) -- (0,4) -- (1.9,2.1);
\draw (2.9,2.9) -- (1.1,1.1);
\draw (.9,2.9) -- (0.5,2.5) --(.5,1.5)  -- (1.5,.5) -- (2.5,.5) -- (2.9,.9);
\draw (3.1,1.1) -- (3.5,1.5) -- (3.5,2.5) -- (2.5,3.5) -- (1.5,3.5) -- (1.1,3.1);
\draw (2.1,1.9) -- (3.5,.5) -- (4.5,.5) -- (4.9,.9);
\draw (5.1,1.1) -- (5.5,1.5) -- (5.5,2.5) -- (4.5,3.5) -- (3.5,3.5) -- (3.1,3.1);
\draw (5.1,3.1) -- (6,4) -- (7,4) -- (7,0) -- (6,0) -- (4.5,1.5) -- (4.5,2.5) -- (4.9,2.9);
\end{scope}
\end{tikzpicture}
\end{center}
\caption{A graph $G$ and the alternating diagram $D_L$ of a link with three components}
\label{Medial Graph}
\end{figure}

\subsection{The Bollob\'as-Riordan-Tutte polynomial}
Recall that, for a graph which is cellularly embedded on the sphere, the Tutte polynomial counts the number of components of the medial graph, and this data provides an easy way to count equivalence classes in our game for graphs on spheres. 

Our goal in this subsection is to show the Bollob\'as-Riordan-Tutte (BRT) polynomial \cite{BR:BRTorientable}  for graphs with a cellular embedding on an orientable surface counts the number of components in the medial graph. We will later show that the number of equivalence classes in our game on arbitrary surfaces unfortunately does not have a simple dependency on this count.

We begin with a definition of the BRT polynomial. Let $v(G), e(G), f(G)$, and $k(G)$ be the number of vertices, edges, faces and connected components of a graph embedded on a surface. The nullity $n(G)$ is $n(G) = e(G)-v(G)+k(G)$, and the genus $g(G)$ is
$$g(G)=\frac{2 k(G)-v(G)+e(G)-f(G)} 2.$$

The BRT polynomial $BRT_G(x,y,z)$ is defined as:
$$BRT_G(x,y,z)= \sum_{H \subseteq G} x^{k(H)-k(G)} y^{n(H)} z^{g(H)}.$$
Here the sum is over all spanning subgraphs of $G$. We slightly departed from the original definition by changing the variable from $x$ to $x+1$. The relation to the Tutte polynomial is given \cite{BR:BRTorientable} by
$$BRT_G(x-1,y-1,1) = T_G(x,y).$$
Chmutov and Pak \cite{ChmutovPak:BollobasRiordan} showed that the Kauffman bracket of an alternating virtual link can be interpreted as an evaluation of the BRT polynomial.
In \cite{DFKLS:GraphsOnSurfaces} it is shown that the Jones polynomial of an arbitrary link can be seen as an evaluation of the BRT polynomial of a graph embedded on a surface that can be constructed from a plane diagram of the link. 
We will use those relations between link diagrams and the BRT polynomial below.

For plane graphs it is known that the Tutte polynomial determines the number of components of the medial graph (e.g. \cite{GodsilRoyle:AlgebraicGraphTheory}). For a graph on an arbitrary orientable surface we need the extension to the BRT polynomial:

\begin{theorem} \label{Thm:ComponentsMedialGraph}
Consider $BRT_G(x,y,z)$ for a graph embedded on an orientable surface $\Sigma$. Construct the link diagram $D_L$ on the surface $\Sigma$ with checkerboard graph $G$ by reversing the arrow in (\ref{CrossingToEdge}). Then the number of components $c$ of $D_L$ is counted by the BRT polynomial as:
$$|BRT_G(-2,-2,1/4)| = 2^{c-1}.$$ 
\end{theorem} 
\begin{proof}
As in \cite{ChmutovPak:BollobasRiordan, DFKLS:GraphsOnSurfaces, DFKLS:DDD} we consider the specialization of the BRT polynomial given by $(x,y,z)=(-A^4-1, -1-A^{-4}, 1/(-A^2-A^{-2})^2).$ 
Recall the slight change of variables from $x$ to $x+1$ from the original definition of the BRT polynomial. 
This defines a Laurent polynomial $\langle D_L \rangle$ in $A$ and $A^{-1}$ by
$$A^{-e(G)} \langle D_L \rangle = A^{2-2 v(G)} BRT_G(-A^4-1, -1-A^{-4}, 1/(-A^2-A^{-2})^2).$$

Thus,
$$\langle D_L \rangle = \sum_{H \subseteq G} A^{e(G)-2 e(H)} (-A^2-A^{-2})^{f(H)-1}.$$

Locally with respect to the link diagram $D_L$ the Laurent polynomial $\langle D_L \rangle$ has the following property \cite{ChmutovPak:BollobasRiordan}:

\begin{equation}
\label{skein relation}
\left \langle
\begin{tikzpicture}[scale=0.4, baseline=9.0]
\draw (0,0)--(2,2);
\draw (2,0)--(1.1,0.9);
\draw (0,2)--(0.9,1.1);
\end{tikzpicture}
\right \rangle=
A \left \langle
\begin{tikzpicture}[scale=0.4, baseline=9.0]
\draw (0,0)to [bend right] (0,2);
\draw (2,0) to [bend left] (2,2);
\end{tikzpicture}
\right \rangle+
A^{-1} \left \langle
\begin{tikzpicture}[scale=0.4, baseline=9.0]
\draw (0,0)to [bend left] (2,0);
\draw (0,2) to [bend right] (2,2);
\end{tikzpicture}
\right \rangle,
\end{equation}
and $$\left \langle
\begin{tikzpicture}[scale=0.4, baseline=-2.0]
\draw (0,0) circle (1cm);
\end{tikzpicture} 
\right \rangle=1.
$$

We are interested in its value at $A=1$. 

We will show that if $L$ has $c$ components, then $\langle D_L \rangle_{A=1} = (-1)^{\nu} (-2)^{c-1},$ where $\nu$ is the number of crossings in the diagram $D_L$. In particular this holds for a diagram of the unknot without crossings.
By induction on the number of crossings, if the two strands in the link diagram on the left-hand side of Equation (\ref{skein relation}) are on two different components of the link, then the two links on the right-hand side have one component less, and we verify:
$$(-1)^{\nu} (-2)^{c-1}=(-1)^{\nu-1} \left ( (-2)^{c-2} + (-2)^{c-2} \right ).$$
If the two strands on the left-hand side of Equation (\ref{skein relation}) are on the same component of the link then one of the links on the right-handside has one component more, while the other one has the same number of components. This verifies:
$$(-1)^{\nu} (-2)^{c-1}=(-1)^{\nu-1} \left ( (-2)^{c-1} + (-2)^{c} \right ).$$
\end{proof}

\subsection{Example}

The graph on the torus in Figure \ref{fig:TicTacToeGraph} has BRT polynomial \cite{vAult:BRTSage}:
\begin{eqnarray*}
BRT_G(x,y,z)&=&x^3+4 x^2 y+9 x^2+6 x y^2+36 x y+32 x+2
   y^4+16 y^3+60 y^2+112 y+48\\
   && +
   \left(2 x y^3+8 x y^2+y^6+9
   y^5+34 y^4+68 y^3+64
   y^2\right) z
\end{eqnarray*}

We have $BRT_G(-2,-2,1/4)=-4$. Therefore, by Theorem \ref{Thm:ComponentsMedialGraph} the medial graph has $3$ components.

\subsection{Components of the medial graph as cycles} \label{Space P}

Each cycle in $G$ (or $G^*$) corresponds to a unique vector in $\GF^{|E|}$. Given such a cycle, the associated vector has a 1 in the position of each edge occurring an odd number of times in the cycle and zeroes elsewhere. Similarly, each subset of $E$ (or $E^*$) will be identified with the vector in $\GF^{|E|}$ having a 1 in the position of an edge if and only if it is present in the subset. Using these two constructions, we conflate the notions of cycles in $G$ and $G^*$ and subsets of $E$ and $E^*$ with their corresponding vectors in $\GF^{|E|}$. 

Let $k_1, \dots, k_c$ be the components of the medial graph for $G$. Since the medial graph of a graph $G$ equals the medial graph of its dual $G^*$, the edges intersected when tracing along a component $k_i$ of the medial graph form a cycle $v_i$ in both $G$ and $G^*$. Note that an edge could be intersected by a component twice, and in that case, the edge would appear twice in the cycle (and thus would have a 0 in that edge's entry of $v_i\in \GF^{|E|}$). Then for $i=1,\ldots, c$, we conclude $v_i \in U^{\perp} \cap (U^{*})^{\perp},$ where $U^{\perp}$ is the cycle space of $G$ and $(U^*)^{\perp}$ is the cycle space of $G^*$.

The following is a straightforward generalization from results for plane graphs (see Section 17.3. of \cite{GodsilRoyle:AlgebraicGraphTheory}) :

\begin{lemma} \label{Generators for P}
Let $G=(V,E)$ be a connected graph embedded on an oriented surface. Let $k_1, \dots, k_c$ be the components of the medial graph, with vectors $v_1, \dots, v_c$ as above. Then
\begin{enumerate}[(i)]
\item Each edge of $G$ occurs in an even number of the vectors $v_i$, and no proper subset of the set $\{ v_1, \dots, v_c \}$ covers every edge in $G$ an even number of times.
\item Each vector is the sum of all of the other vectors.
\item The subspace $\mathcal P$ of $\GF^{|E|}$ generated by $\{ v_1, \dots, v_c \}$ has dimension $c-1$.
\end{enumerate}
\end{lemma}

\section{Plane graphs}
\label{section:plane}

The linear algebra discussed in the previous sections is particularly nice for graphs embedded on spheres (i.e. plane graphs). Since every cycle in such a graph bounds a disk, it follows that $U^*=U^{\perp}$, and thus $U \cap U^*=U\cap U^{\perp}$. In other words, for graphs on spheres $U \cap U^*$ is the bicycle space.

The following theorem gives an easy way to count the number of equivalence classes in our game for graphs on spheres. The two moves on bicolorings described here are more complicated than the ones given in the introduction. This is due to the fact that we allow our graphs to have loops and bridges. 

\begin{theorem} \label{Thm:PlaneGraphs}
Let $G=(V,E)$ be a finite, connected plane graph. The following two moves define an equivalence relation on the set of edge bicolorings of $G$. Two colorings are equivalent if and only if one can be obtained from the other by a finite sequence of the following two moves:
\begin{enumerate}
\item Around a vertex, switch the colors of all edges with exactly one endpoint at that vertex (i.e. all non-loops adjacent to the vertex).
\item Around a face of the embedded graph, switch all colors of the edges appearing exactly once in the boundary cycle of the face.
\end{enumerate}
Then the absolute value of the Tutte polynomial evaluation $|T(-1,-1)|$ yields the number of equivalence classes of colorings.
\end{theorem}

\begin{proof}
As in the example in Section \ref{section:example},
 the problem reduces to computing the intersection of the cycle and the cocycle spaces of $G$. Using Euler characteristic, the number of equivalence classes in our game is $2^b$ where $b=|E| - \dim{(U+U^*)} = |E|-(\dim U+\dim U^*-\dim{(U\cap U^*)}) = 2g+\dim{(U\cap U^*)}$ where $g$ is the genus of the surface on which the graph is embedded. In the plane case then, the number of equivalence classes is simply $2^b$ where $b=\dim{(U\cap U^*)}$.
 
The row vectors of the incidence matrix for $G$ span the cocycle space $U\subseteq \GF^{|E|}$ which has dimension $|V|-1$ (see e.g. \cite{GodsilRoyle:AlgebraicGraphTheory}, Lemma 14.15.1).  The cocycle space $U^*$ of the dual graph $G^*=(V^*,E^*)$ has dimension $|V^*|-1$. Using Euler characteristic and duality, we see $\dim U^*=|V^*|-1=|E|-|V|+1 = |E|-\dim U$. Since $\dim U + \dim U^{\perp}=|E|$ and $U^*\subseteq U^{\perp}$,  it follows that $U^*=U^{\perp}$ and $U \cap U^*=U \cap U^{\perp}$. The claim follows since the dimension of the bicycle space $U\cap U^{\perp}$ is given by $b$ where $2^b=|T(-1,-1)|$ \cite{RosenstiehlRead:EdgeTripartition}. 
\end{proof}

Given a knot diagram on an orientable surface, a region crossing change (RCC) is the result of reversing all crossings incident to a region of the diagram. Shimizu \cite{Shimizu:RegionChanges} and Cheng-Gao \cite{ChengGao:RegionChanges} study equivalence classes of planar link diagrams under the RCC operation. Their results follow as a corollary of Theorem \ref{Thm:PlaneGraphs}.
\begin{corollary}
Suppose $D$ is a $c$ component link diagram in the plane. Then the number of equivalence classes of diagrams with the same shadow as $D$ under the RCC operation is $2^{c-1}$. In particular, every knot diagram can be transformed to one of the unknot via a finite sequence of RCC moves.
\end{corollary}
\begin{proof}
Given a link diagram in the plane, consider either of its checkerboard graphs. By Theorem \ref{Thm:ComponentsMedialGraph}, the Tutte polynomial of this graph specializes to $|T(-1,-1)|=2^{c-1}$, and the result now follows by  Theorem \ref{Thm:PlaneGraphs}.
\end{proof}

\subsection{Representatives for the equivalence classes}

For connected plane graphs there is a natural way of choosing representatives for the equivalence classes of the color change game using the vectors defined in Section \ref{Space P}. 


Say $G=(V,E)$ is a finite, connected plane graph with medial graph having $c$ components. In Lemma \ref{Generators for P}, we described a $c-1$ dimensional space $\mathcal{P} \subseteq U^{\perp} \cap (U^*)^{\perp}$. In the plane setting since $U^*=U^{\perp}$ for plane graphs, the spaces $U\cap U^{\perp}$, $U\cap U^*$, and $\mathcal{P}$ are all the same. (As we will see in the next section, this is not true for arbitrary surfaces.) The vectors $v_1, \ldots, v_c$ described in Section \ref{Space P} span $\mathcal{P}$, and by Lemma 17.3.3 in \cite{GodsilRoyle:AlgebraicGraphTheory} any subset of $c-1$ of these $c$ vectors is a basis.

\begin{theorem} 
Let $v_1, \dots, v_{c-1}$ be the basis for the bicycle space for $G$ described above. For  $j=1, \dots, c-1,$ pick an edge $e_j\in E$ such that $v_j$ has a 1 in the position of $e_j$ and $v_k$ has 0 in the position of $e_j$  for all $k<j$. (Such a set ${\mathcal S} = \{e_1, \dots, e_{c-1}\}$ can be found by Lemma 17.3.2 in \cite{GodsilRoyle:AlgebraicGraphTheory}.) Then the set ${\mathcal S}$ forms a basis for $\GF^{|E|}/(U+U^*).$
\end{theorem}
\begin{proof}
By construction, for each $j \in \{1, \dots, c-1 \}$, the vector $e_j$ is not perpendicular to the bicycle $v_j$. Thus \cite{GodsilRoyle:AlgebraicGraphTheory} :
$$e_j \not \in (U \cap U^*)^{\perp} = U^{\perp} + (U^*)^{\perp} = U^* + U.$$ 
Also by construction, any linear combination $$\sum_{j=1}^{c-1} \alpha_j e_j, \quad \mbox{with} \quad \alpha_i \in \GF$$ is not contained in $U + U^*$ either, and the vectors in $\mathcal S$ are linearly independent in, and form a basis for $$\GF^{|E|}/(U+U^*).$$
\end{proof}

\section{Graphs on orientable surfaces}
\label{section:surfaces}

Let $G=(V,E)$ be a graph with a cellular embedding on an orientable surface $\Sigma$ of genus $g$. We want to compute the number of equivalence classes of bicolorings of $G$ under the two moves described in Theorem \ref{Thm:PlaneGraphs}. While the general solution to this problem is not as simple as the genus 0 case, we can obtain an answer by looking at certain maps on homology.

Since $G$ is cellularly embedded on $\Sigma$, so is its dual graph $G^*=(V^*,E^*)$. The edges in $E^*$ are in $1$-$1$-correspondence to the edges of $E$. A face of $G^*$ corresponds to a vertex $v$ of $G$ and is bounded by edges in $E^*$ that correspond to edges in $E$ adjacent to $v$.  Assume a fixed ordering of $E$, and the induced order on $E^*$. 
The cocycle space $U$ of $G$ has dimension $|V|-1$, and the cocycle space  $U^*$ of $G^*$ has dimension $|V^*|-1$. As discussed in the example in Section \ref{section:example}, the cocycle space $U^*$ is perpendicular to $U$, so $U^* \subseteq U^{\perp}.$ The dimension of $U^{\perp}$ is $|E|-\dim U.$ By assumption on the Euler characteristic of the surface we
have $$|V|-|E|+|V^*| = 2 - 2 g.$$ Therefore, the quotient space $U^{\perp}/U^*$ has dimension $2g$. In fact, the space $U^{\perp}/U^*$ is the first homology group $H_1(\Sigma, \GF)$ of $\Sigma$ with coefficients in the field $\GF$. 

The map $\varphi: U^{\perp} \longrightarrow H_1(\Sigma, \GF)$
can be constructed as follows (\cite{BCFN:Homology}, and compare with \cite{Eppstein:EmbeddedGraphs}): Chose a spanning tree $T$ of $G$ and a spanning tree $C$ of $G^*$ such that none of the edges of $C$ are dual to the edges of $T$. Such a spanning tree $C$ is called a co-tree for $T$. There are exactly $2g$ edges of $G^*$ that are neither in $C$ nor dual to edges in $T$; denote these by $e_1^*, \dots, e_{2g}^*$. 

For $j=1, \dots, 2g,$ let $p_j\in\GF^{|E|}$ be the unique cycle in $C \cup e_j^*.$ Then for a vector $u \in U^{\perp}$ the image $\varphi (u)$ can be expressed by \cite{BCFN:Homology}:
$$\varphi(u) = ( \langle p_1, u \rangle, \dots, \langle p_{2g}, u \rangle ),$$ where $\langle \cdot , \cdot \rangle$ is the inner product in $\GF^{|E|}$. By construction, we have $\ker \varphi = U^*$.
Using the same construction but switching the roles of $G$ and $G^*$, we define a map $$\varphi^*:(U^*)^{\perp} \rightarrow H_1(\Sigma, \GF).$$ In this case $\ker \left(\varphi^*\right) = U$.

\subsection{The intersection $U \cap U^{*}$ as a subspace of ${\mathcal{P}}$}

We once again return to the space $\mathcal P$ from Section \ref{Space P}. If $G$ has a medial graph with $c$ components, the space $\mathcal P$ is generated by the vectors $v_1, \ldots, v_c$ which are the cycles in $G$ and $G^*$ traced out when traversing components of the medial graph. By Lemma \ref{Generators for P}, the space $\mathcal{P}$ is a subspace of $U^{\perp}\cap (U^*)^{\perp}$ of dimension $c-1$.

We will give a proof of a theorem of Lins, Richter and Shank \cite {LinsEtAl:GaussCode} by adapting elegant ideas of Lamey, Silver and Williams \cite{LSW:Bicycles}. 

\begin{theorem}[\cite{LinsEtAl:GaussCode}]
$U \cap U^{*} \subseteq {\mathcal{P}}\subseteq U^{\perp} \cap (U^*)^{\perp}.$
\end{theorem}
\begin{proof}
The inclusion $\mathcal{P} \subseteq U^{\perp} \cap (U^*)^{\perp}$ is by definition of $\mathcal P$.
Let $u \in \GF^{|E|}$ be in $U$, and recall that $U^{\perp}$ is the space of all cycles in the graph $G$. Thus, we have 
\begin {equation} \label{cycle condition}
\langle c, u \rangle =0 \mbox{ for all cycles }  c \mbox{ in } G.
\end{equation}
If $u \in U \cap U^*$ then additionally 
\begin{equation}
\langle c^*, u \rangle =0 \mbox{ for all cycles } c^* \mbox{ in } G^*.
\end{equation}

As in \cite{LSW:Bicycles} this allows us to construct for $u \in U \cap U^*$ a pair $(v, v^*)=(v(u), v^*(u))$ of vectors in $\GF^{|V|} \times \GF^{|V^*|}$ as follows. Fix a vertex $x_0$ in $G$ and an adjacent vertex $x_0^*$ in $G^*$, and assign values in $\GF$ to these two vertices. To simplify the notation we will denote these values by $x_0$ and $x_0^*$ as well. To obtain an assignment to any other vertex $x_{k}$ in $G$ choose a path $x_0. \dots, x_i, x_{i+1}, \dots, x_k$ from $x_0$ to $x_{k}$.
The value of $x_{i+1}$ is determined by $x_{i+1}=x_i+u_i$, where $u_i$ is the edge assignment given by $u=(u_i)_{i=1, \dots, |E|}$, see Figure \ref{vertex assignments}. 

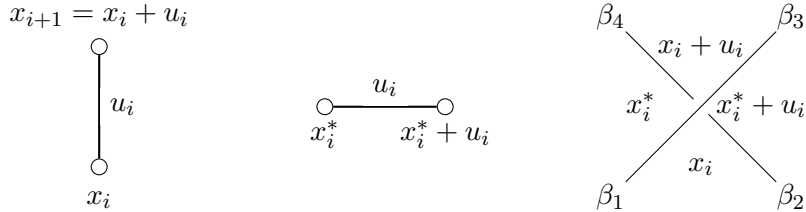
\begin{figure}[h]
\begin{center}
\begin{tikzpicture}[baseline=25]
\vertex [label={$x_{i+1}=x_i+u_i$}](p1) at (1,1.8) {};
\vertex [label={[yshift=-0.8cm] $x_{i}$}](p2) at (1,0.2) {};
\Edge[label={$u_i$}, labelstyle=right](p1)(p2);
\begin{scope}[xshift=3.8 cm, baseline=25]
\vertex [label={[yshift=-0.8cm] $x_i^*+u_i$}](p1) at (1.8, 1) {};
\vertex [label={[yshift=-0.8cm] $x^*_{i}$}](p2) at (0.2, 1) {};
\Edge[label={$u_i$}, labelstyle=above](p1)(p2);
\end{scope}
\begin {scope}[xshift=8cm, baseline = 25]
\draw (0,0)--(2,2);
\draw (1,1.8) node{$x_i+u_i$};
\draw (1,0.2) node{$x_i$};
\draw (1.8,1) node{$x_i^*+u_i$};
\draw (0.2,1) node{$x_i^*$};
\draw (2,0)--(1.1,0.9);
\draw (0.90, 1.1)--(0,2);
\draw (-0.2,-0.2) node{$\beta_1$};
\draw (2.2, -0.2) node{$\beta_2$};
\draw (-0.2, 2.2) node{$\beta_4$};
\draw (2.2, 2.2) node{$\beta_3$};
\end{scope}
\end{tikzpicture}
\end{center}
\caption{An edge in the graph, in its dual graph and the corresponding crossing in the medial graph. The assignments to vertices and faces are: $x_{i+1}=x_i+u_i, x^*_{i+1}=x^*_i+u_i, \beta_1=\beta_3=x_i+x_i^*$ and $\beta_2=\beta_4=x_i+x_i^*+u_i$} \label{vertex assignments}
\end{figure}

By Equation (\ref{cycle condition}) these assignments to the vertices do not depend on the chosen path. Similarly, the assignments to the vertices of $G^*$ are determined by $u$, see Figure \ref{vertex assignments}. The assignments to the vertices in the graph and the dual graph give us assignments to the faces of the medial graph. The next step is to assign values $\beta_i$ to the arcs of the medial graph
by taking the sum of the values of the adjacent faces, see Figure \ref{vertex assignments}. We see that this is well-defined, i.e. $\beta_1=\beta_3$, but also that the values of the two under-arcs are equal: $\beta_2=\beta_4$. This means that along a component of the medial graph all arcs are assigned the same value, and this assignment yields an element of ${\mathcal P}$. The last step is to express the initial vector $u$ as an element of 
${\mathcal P}$. At the crossing the sum of the values assigned to the two components is $\beta_1+\beta_2=x_i+x_i^*+x_i+x_i^*+u_i=u_i$. Thus the sum, over all crossings, of the components of the link with multiplicities given by their assigned values is $u$. 
\end{proof}

Recall the maps $\varphi: U^{\perp} \longrightarrow H_1(\Sigma, \GF)$ and  $\varphi^*: (U^*)^{\perp} \longrightarrow H_1(\Sigma, \GF)$ from the beginning of this section.  The kernel of $\varphi$ is $U^*$, and the kernel of $\varphi^*$ is $U$. Since $\mathcal{P} \subseteq U^{\perp} \cap (U^*)^{\perp}$, we can restrict both $\varphi$ and $\varphi^*$ to the common domain $\mathcal{P}$.  
In general if an element of $U^{\perp} \cap (U^*)^{\perp}$ is in the kernel of $\varphi$ it is not necessarily in the kernel of $\varphi^*$. For example, in the graph depicted in Figure \ref{fig:TicTacToeGraph} the edges $\{e_1,e_2,e_3,e_4\}$ form a cycle both in the graph and its dual graph. This cycle is in the kernel of $\varphi$ but not $\varphi^*$.

The elements in $\mathcal P$, however, are represented by sums of cycles which are homologous in $G$ and $G^*$, and the maps $\varphi$ and $\varphi^*$ restricted to $\mathcal P$ differ at most by a change of basis. Hence, $\ker \left(\varphi |_{\mathcal{P}}\right)$ and $\ker \left(\varphi^* |_{\mathcal{P}}\right)$ both consist of the nulhomologous elements of $\mathcal{P}$. Thus, $\ker \left(\varphi |_{\mathcal{P}}\right) = \ker \left(\varphi^* |_{\mathcal{P}}\right) = U\cap U^*$, and we get the following theorem.

\begin{theorem}
Let $G = (V,E)$ be a finite, cellularly embedded graph on a closed, orientable surface $\Sigma$ of genus $g$. Let $\mathcal{P}$ and $\varphi$ be as described above with $b = \dim \ker \left(\varphi |_\mathcal{P}\right)$. Then the number of equivalence classes in the color changing game for $G$ on $\Sigma$ is $2^{2g + b}$.
\end{theorem}


It is an interesting question to more directly relate these results to the study of RCC equivalence of cellularly embedded, checkerboard colorable link diagrams on surfaces. We conclude with an example highlighting our results.

\subsection{Example}
Consider the following example of a graph on six vertices and eight edges embedded on a torus:

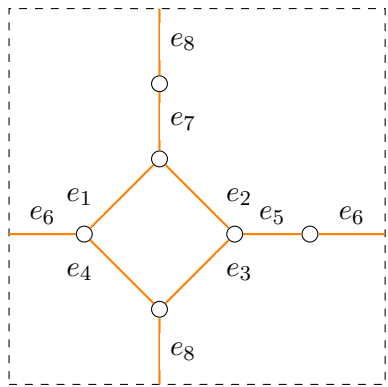
\begin{center}
\begin{tikzpicture}[scale=0.5]
 \GraphInit[vstyle=Art]
\draw [dashed] (-2,-4) -- (8,-4) -- (8,6) -- (-2,6) -- (-2,-4);
\vertex (p1) at (0,0) {};
\vertex (p2) at (2,2) {};
\vertex (p3) at (4,0) {};
\vertex (p4) at (2,-2) {};
\vertex (p5) at (6,0) {};
\vertex (p6) at (2,4) {};
\Edge [label=$e_1$, labelstyle={left=7pt}](p1)(p2);
\Edge [label=$e_2$, labelstyle={right=7pt}](p2)(p3);
\Edge [label=$e_3$, labelstyle={right=7pt}](p3)(p4);
\Edge [label=$e_4$, labelstyle={left=7pt}](p4)(p1);
\Edge [label=$e_5$, labelstyle=above](p3)(p5);
\Edge [label=$e_6$, labelstyle=above](p5)(8,0);
\Edge [label=$e_6$, labelstyle=above](-2,0)(p1);
\Edge [label=$e_7$, labelstyle=right](p2)(p6);
\Edge [label=$e_8$, labelstyle=right](p6)(2,6);
\Edge [label=$e_8$, labelstyle=right](p4)(2,-4);
\end{tikzpicture}
\end{center}

The cocycle space  $U$ is generated by the row vectors of the incidence matrix 
$$ \mathcal{I}=\left(
\begin{array}{cccccccc}
 1 & 0 & 0 & 1 & 0 & 1 & 0 & 0 \\
 1 & 1 & 0 & 0 & 0 & 0 & 1 & 0 \\
 0 & 1 & 1 & 0 & 1 & 0 & 0 & 0 \\
 0 & 0 & 1 & 1 & 0 & 0 & 0 & 1 \\
 0 & 0 & 0 & 0 & 1 & 1 & 0 & 0 \\
 0 & 0 & 0 & 0 & 0 & 0 & 1 & 1 \\
\end{array}
\right),
$$
and it has dimension $5$.

The cocycle space of the dual graph is generated by the row vectors of the incidence matrix of the dual graph
$$ \mathcal{I}^*=
\left(
\begin{array}{cccccccc}
 1 & 1 & 1 & 1 & 0 & 0 & 0 & 0 \\
 1 & 1 & 1 & 1 & 0 & 0 & 0 & 0 \\
\end{array}
\right),
$$
of dimension $1$.

The Bollob\'as-Riordan-Tutte polynomial of the graph is \cite{vAult:BRTSage}:
$$BRT(x,y,z)=x^5+8 x^4+28 x^3+5 x^2 y+56 x^2+2 x y^2+20 x y+65 x+y^3 z+4 y^2 z+4
   y^2+26 y+36,$$
   and its value $|BRT(-2,-2,1/4)|=|-8|=2^{4-1}$. Hence the medial graph has four components.
  
The vector space $\mathcal P$ of dimension 3 is generated by the vectors $v_1, v_2, v_3, v_4$ corresponding to the cycles traced out by the four components of the medial graph. These vectors are the rows of the following matrix.

$$\left(
\begin{array}{cccccccc}
 1 & 1 & 0 & 0 & 1 & 1 & 0 & 0 \\
 0 & 0 & 1 & 1 & 1 & 1 & 0 & 0 \\
 0 & 1 & 1 & 0 & 0 & 0 & 1 & 1 \\
 1 & 0 & 0 & 1 & 0 & 0 & 1 & 1 \\
\end{array}
\right).
$$

We are interested in $\ker \varphi |_{\mathcal P},$ where $\varphi: U^{\perp} \longrightarrow H_1(\Sigma,\GF).$
For that we fix a spanning tree $T$ of the graph, say $T=\{e_1,e_3,e_4,e_5,e_7\}$ and a co-tree $C$ in the dual graph, say $C=\{e_2^* \}$. Thus, $\{ e_6^*, e_8^* \}$ are the edges neither in $C$ nor dual to edges in $T$. The unique cycle in $C \cup \{e_6^* \}$ is $\{e_6^* \}$, and the unique cycle in $C \cup \{ e_8^* \}$ is $\{e_8^* \}$. So we obtain the following vectors $p_1$ and $p_2$ which are the row vectors of the following matrix.

$$\left(
\begin{array}{cccccccc}
 0 & 0 & 0 & 0 & 0 & 1 & 0 & 0 \\
 0 & 0 & 0 & 0 & 0 & 0 & 0 & 1 \\
\end{array}
\right).
$$

The image of the vector space $\mathcal P$ in $H_1(\Sigma,\GF)$ is determined by:

$$\left(
\begin{array}{cccccccc}
 1 & 1 & 0 & 0 & 1 & 1 & 0 & 0 \\
 0 & 0 & 1 & 1 & 1 & 1 & 0 & 0 \\
 0 & 1 & 1 & 0 & 0 & 0 & 1 & 1 \\
 1 & 0 & 0 & 1 & 0 & 0 & 1 & 1 \\
\end{array}
\right) 
\left(
\begin{array}{cc}
 0 & 0 \\
 0 & 0 \\
 0 & 0 \\
 0 & 0 \\
 0 & 0 \\
 1 & 0 \\
 0 & 0 \\
 0 & 1 \\
\end{array}
\right) =
\left(
\begin{array}{cc}
 1 & 0 \\
 1 & 0 \\
 0 & 1 \\
 0 & 1 \\
\end{array}
\right).
$$  
Thus the image $\varphi({\mathcal P})$ is 2-dimensional. Since $\dim {\mathcal P}=3$ this implies $\dim \ker \varphi |_{\mathcal P}=1$. Finally, since $\ker \varphi |_{\mathcal P} = U \cap U^*$ we conclude the dimension of $\GF^8/(U+U^*)$ is:

\begin{eqnarray*}
|E| -\dim (U+U^*) &=& 2g+  \dim (U \cap U^*)\\
&=& 2g + \dim{(\ker \varphi |_{\mathcal P})} \\
&=& 2+1\\
&=&3.
\end{eqnarray*}

\bibliography{2014Dimers}
\bibliographystyle {amsalpha}

\end{document}